\documentclass[12pt]{article}

\usepackage{amsthm, amssymb, amstext}
\usepackage[fleqn]{amsmath}
\usepackage{latexsym}
\usepackage{pgfpages,tikz,xspace}

\def\indr{\mc{I}^{(r)}(G)}

\def\del{\!\downarrow\!}
\def\A{\mathcal{A}}
\def\B{\mathcal{B}}
\def\C{\mathcal{C}}

\def\I{\mathcal{I}}

\def\M{\mathcal{M}}
\def\N{\mathcal{N}}

\def\mc{\mathcal}

\newtheorem{lemma}{Lemma}[section]

\newtheorem{theorem}{Theorem}[section]

\newtheorem{conjecture}{Conjecture}[section]

\newtheorem{claim}{Claim}

\theoremstyle{definition}

\newtheorem*{remark*}{Remark}

\title{Erd\H{o}s-Ko-Rado Theorems\\ for a Family of Trees}

\author{Carl Feghali, Matthew Johnson and Daniel Thomas \\ School of Engineering and  Computing Sciences, \\ Durham University\\
{\small \texttt{\{carl.feghali,matthew.johnson2,d.j.r.thomas\}@durham.ac.uk}}
}

\date{}

\begin{document}

\maketitle

\begin{abstract}
Given a graph $G$ and an integer $r\geq 1$, let $\I^{(r)}(G)$ denote the family of independent sets of size $r$ of $G$. For a vertex $v$ of $G$, let $\I^{(r)}_v(G)$ denote the family of independent sets of size $r$ that contain~$v$.  This family is called an $r$-star and $v$ is the centre of the star. Then $G$ is said to be $r$-EKR if no pairwise intersecting subfamily of $ \I^{(r)}(G)$ is bigger than the largest $r$-star, and if every maximum size pairwise intersecting subfamily of $ \I^{(r)}(G)$ is an $r$-star, then $G$ is said to be strictly $r$-EKR. Let $\mu(G)$ denote the minimum size of a maximal independent set of $G$. Holroyd and Talbot conjectured that if $2r \leq \mu(G)$, then $G$ is $r$-EKR and strictly $r$-EKR if $2r < \mu(G)$.  

An elongated claw is a tree in which one vertex is designated the root and no vertex other than the root has degree greater than 2.  A depth-two claw is an elongated claw in which every vertex of degree~1 is at distance 2 from the root. We show that  if $G$ is a depth-two claw, then $G$ is strictly $r$-EKR if $2r \leq \mu(G)+1$, confirming the conjecture of Holroyd and Talbot for this family. We also show that if $G $ is an elongated claw with $n$ leaves and at least one leaf adjacent to the root, then $G$ is $r$-EKR if $2r \leq n$.  

Hurlbert and Kamat had conjectured that one can always find a largest $r$-star of a tree whose centre is a leaf.  Baber and Borg have separately shown this to be false.   We show that, moreover, for all $n \geq 2$, $d \geq 3$, there exists a positive integer $r$ such that there is a tree where the centre of the largest $r$-star is a vertex of degree $n$ at distance at least $d$ from every leaf. 
\end{abstract}

\section{Introduction}
In this paper, we consider graph theoretic versions of the following famous result of Erd\H{o}s, Ko and Rado~\cite{erdos};  the characterization of the extremal case was provided by Hilton and Milner~\cite{hilton}.
\bigskip

\noindent {\bf EKR Theorem} (Erd\H{o}s, Ko, Rado~\cite{erdos}; Hilton, Milner~\cite{hilton}) \emph{Let $n$ and $r$ be positive integers, $n \geq r$, let $S$ be a set of size $n$ and let ${\mathcal A}$ be a family of subsets of $S$ each of size $r$ that are pairwise intersecting.   If $n \geq 2r$, then
\[
|\mathcal A| \leq \binom{n-1}{r-1}. 
\]
Moreover, if $n > 2r$ the upper bound is attained only if  the sets in $\A$ contain a fixed element of $S$.}

\bigskip

Throughout this paper, graphs are simple and undirected. Let $K_n$ denote the complete graph on $n$ vertices, and let $K_{1, n}$ denote a claw. An independent set in a graph is a set of pairwise non-adjacent vertices.  Let $\mu(G)$ denote the minimum size of a maximal independent set in $G$. 

Given a graph $G$ and an integer $r\geq 1$, let $\I^{(r)}(G)$ denote the family of independent sets of $G$ of cardinality $r$. For a vertex $v$ of $G$, let $\I^{(r)}_v(G)$ be the subset of $\I^{(r)}(G)$ containing all sets that contain $v$.  This is called an \emph{$r$-star} (or just star) and $v$ is its \emph{centre}.  We say that $G$ is $r$-EKR if no pairwise intersecting family $\A \subseteq \I^{(r)}(G)$ is larger than the biggest $r$-star, and strictly $r$-EKR if every pairwise intersecting family that is not an $r$-star is smaller than the the largest $r$-star of  $\I^{(r)}(G)$.

The EKR Theorem can be seen as a statement about the maximum size of a family of pairwise intersecting independent sets of size $r$ in the empty graph on $n$ vertices.  We quickly obtain another formulation of the EKR Theorem by noting that an independent set of the claw that contains more than one vertex contains only leaves.

\begin{theorem}\label{starekr}
Let $n$ and $r$ be positive integers, $n \geq r$.  The claw $K_{1, n}$ is $r$-EKR if $n \geq 2r$ and strictly $r$-EKR if $n > 2r$.
\end{theorem}

There exist EKR results for several graph classes.  The reader is referred to~\cite{borg2} and the references therein for further details.

The following was conjectured by Holroyd and Talbot~\cite{talbot}. 

\begin{conjecture}[Holroyd, Talbot~\cite{talbot}]\label{c1}
Let $r$ be a positive integer and let $G$ be a graph. Then $G$ is $r$-EKR if $\mu(G) \geq 2r$ and strictly $r$-EKR if $\mu(G) > 2r$. 
\end{conjecture} 

\noindent  This conjecture appears difficult to prove or disprove. It is nevertheless known to be true for many graph classes such as the disjoint union of complete graphs each of order at least two, powers of paths~\cite{holroyd} and powers of cycles~\cite{talbot1}.  See~\cite{borg,spencer1,spencer2,talbot,kamat} for further examples.

A usual technique to prove results of this kind is to find the centre of the largest $r$-star of a graph and this will prove useful to us.  In this vein, Hurlbert and Kamat~\cite{kamat} conjectured the following for the class of trees. 

\begin{conjecture}[Hurlbert, Kamat~\cite{kamat}]\label{c2}
Let $n$ and $r$ be positive integers, $n \geq r$. If $T$ is a tree on $n$ vertices, then there
is a largest $r$-star of $T$ whose centre is a leaf.
\end{conjecture}

\noindent They were able to prove Conjecture~\ref{c2} for $1 \leq r \leq 4$~\cite{kamat}. The conjecture does not, however, hold for any $r\geq 5$. This was shown independently by Baber~\cite{baber} and Borg~\cite{borg3} who gave counterexamples in which the largest $r$-star is centred at a vertex whose degree is $2$.  Moreover, mindful of Conjecture~\ref{c1}, we remark that in each of their counterexamples $G$, the value of $r$ does not  exceed~$\mu(G) / 2$.

\subsection{Results} 

We consider a family of trees called \emph{elongated claws}.  An elongated claw has one vertex that is its \emph{root}.  Every other vertex has degree 1 or 2 (it is possible that the root also has degree 1 or 2).  A vertex of degree 1 is called a \emph{leaf}.  A path from the root to a leaf is a \emph{limb}.  A limb is \emph{short} if it contains only one edge.  If every leaf is distance at 2 from the root (that is, if every limb contains two edges), then the graph is a \emph{depth-two claw}.  

We are now ready to state our main results. 
\begin{theorem}\label{mainthm}
Let $r$ be a positive integer and let $G$ be a depth-two claw. Then $G$ is strictly $r$-EKR if $\mu(G) \geq 2r-1$.
\end{theorem}
\noindent  Theorem~\ref{mainthm} confirms (and is stronger than) Conjecture~\ref{c1} for depth-two claws. 

\begin{theorem}\label{starthm}
Let $n$ and $r$ be positive integers, $n \geq 2r$, and let $G$ be an elongated claw with $n$ leaves and a short limb.  Then $G$ is $r$-EKR.
\end{theorem} 
 
\noindent  Theorem~\ref{starthm} does not confirm (but only supports) Conjecture~\ref{c1} for the class of elongated claws with short limbs since $\mu(G)$ may be much larger than the number of leaves in $G$. 

We remark that similar EKR results (that is, with weaker bounds than that of Conjecture~\ref{c1}) were obtained in~\cite[Theorem 8]{holroyd} and \cite[Proposition 4.3]{wood}. Satisfying the bound of Conjecture~\ref{c1} in Theorem~\ref{starthm}, and in general for elongated claws, is left as an open problem.  

Our final result relates to the problem of trying to find the centre of largest $r$-stars in trees.  We show that such centres can, in some sense, be located anywhere within a tree.

\begin{theorem} \label{centrethm}
Let $n$ and $d$ be positive integers, $n \geq 2$, $d \geq 3$.  Then there exists a positive integer $r$ such that there is
a tree where the centre of the largest $r$-star is a vertex of degree $n$ and at distance at least $d$ from every leaf.
\end{theorem}

In the remaining sections we prove Theorems~\ref{mainthm}, \ref{starthm} and \ref{centrethm}. 
 
\section{Depth-two Claws} 

In this section, we prove Theorem~\ref{mainthm} after first proving a number of useful results.  This next lemma will also be used in the proof of Theorem~~\ref{starthm}.

\begin{lemma}\label{starlm}
Let $r$ be a positive integer, and let $G$ be an elongated claw.   Then there is a largest $r$-star of $G$ whose centre is a leaf.
\end{lemma}

\begin{proof}
Let $v$ be a vertex of $G$ that is not a leaf, and let $L$ be the limb of~$G$ that contains $v$ (if $v$ is the root, then $L$ can be any limb).  Let $x$ be the leaf of $L$.  We find an injection $f$ from $\I^{(r)}_v(G)$ to $\I^{(r)}_x(G)$ which proves that $|\I^{(r)}_x(G)| \geq |\I^{(r)}_v(G)|$ and the lemma immediately follows. 

Let $w$ be the unique neighbour of $x$. Let $A \in \I^{(r)}_v(G)$.  
\begin{enumerate}
\item If $x \in A$, then let $f(A) = A$.
\item If $x \not\in A$ and $w \not\in A$, then let $f(A) = A \backslash\{v\} \cup \{x\}$.
\item If $x \not\in A$ and $w \in A$, then let $X = \{x=x_1, x_2, \dots, x_m = v\}$ be the set of vertices in $L$ from $x$ towards $v$. Let $A \cap X = \{x_{i_1}, \dots, x_{i_j}\} = Y$ for some $m >  j \geq 1$. Let $Z = \{x_{i_1-1}, \dots, x_{i_j-1}\}$. Observe that $|Y| = |Z|$ and $x \in Z$ since $w \in Y$. Then let $f(A) =(A \cup Z)\backslash Y$. 
\end{enumerate}
To prove that $f$ is injective we consider distinct $A_1, A_2 \in \I^{(r)}_v(G)$.  If $f(A_1)$ and $f(A_2)$ are defined by the same case (of the three above), then it is clear that $f(A_1)$ and $f(A_2)$ are distinct.  When they are defined by different cases, we simply note that in the first $f(A)$ always contains $v$, in the second $f(A)$ contains neither $v$ nor any of its neighbours, and in the third $f(A)$ contains a neighbour of $v$.
\end{proof}

We note that Lemma~\ref{starlm} confirms Conjecture~\ref{c2} for elongated claws. 

\begin{remark*}
The property of elongated claws in Lemma~\ref{starlm} is a much weaker version of the \emph{degree sort property}; a graph has this property if the size of an $r$-star centred at $u$ is at least the size of an $r$-star centred at~$v$ whenever the degree of $u$ is less than that of $v$.  Hurlbert and Kamat~\cite{kamat} observed that depth-two claws have this property.  We note that not all elongated claws possess it.  For example, consider an elongated claw with three limbs of lengths 1, 2 and 3.  Then the 4-star centred at the neighbour of the root in the limb of length 3 has size 2, but the 4-star centred at the leaf of the limb of length 2 has size 1.  It remains to determine which elongated claws --- or, more generally, which trees --- have the degree sort property.  We might also ask which trees have the following weaker property: if $i<j$, then the size of the largest $r$-star of all those stars centred at vertices of degree~$i$ is at least  the size of the largest $r$-star of all those centred at vertices of degree~$j$.
\end{remark*}

\begin{lemma}\label{mainstar}
Let $n$ and $r$ be positive integers, $n \geq r$, and let $G$ be a depth-two claw with $n$ leaves. Then the largest $r$-star of $G$ is centred at a leaf and has size
\[
\binom{n-1}{r-1}2^{r-1} + \binom{n-1}{r-2}. 
\]
\end{lemma}

\begin{proof}
By Lemma~\ref{starlm}, there is a largest $r$-star whose centre is a leaf (and clearly, by symmetry, all leaves are equivalent). So let $v$ be a leaf of $G$ and let $c$ be the root of $G$.  Define a partition: $\B = \{B \in \I_v^{(r)}(G): c \not\in B\}$ and $\C = \{C \in \I_v^{(r)}(G): c \in C\}$. Then $|\B| = \binom{n-1}{r-1}2^{r-1}$ since each member of~$\B$ intersects $r-1$ of the $n-1$ limbs that do not contain $v$ and can contain either of the 2 vertices (other than the root) of each of those limbs. And $|\C| = \binom{n-1}{r-2}$ since each member of $\C$ contains $r-2$ of the $n-1$ leaves other than $v$.  The proof is complete. 
\end{proof}

In order to prove Theorem~\ref{mainthm}, we shall need two auxiliary results. 

\begin{theorem}[Meyer~\cite{meyer}; Deza and Frankl~\cite{deza}]\label{thm:meyer}
Let $n$, $r$ and $t$ be positive integers, $n \geq r$, $t \geq 2$, and let $G$ be the disjoint union of $n$ copies of $K_t$.  Then $G$ is $r$-EKR and strictly $r$-EKR unless $r=n$ and $t=2$. 
\end{theorem}

For a family of sets $\A$ and nonnegative integer $s$, the $s$-shadow of $\A$, denoted $\partial_s\A$, is the family $\partial_s\A = \{S: |S| = s, \exists A \in \A, S \subseteq A\}$. 

\begin{lemma}[Katona~\cite{katona}]\label{shadow}
Let $a$ and $b$ be nonnegative integers and let $\A$ be a family of sets of size $a$ such that $|A \cap A'| \geq b \geq 0$ for all $A, A' \in \A$. Then $|\A| \leq |\partial_{a-b}\A|$
\end{lemma}

\noindent The proof of Theorem~\ref{mainthm} is inspired by a proof of the EKR theorem~\cite{frankl}. To the best of our knowledge, the proof is the first to make use of shadows in the context of graphs.

\begin{proof}[Proof of Theorem~\ref{mainthm}]
Let $c$ be the root of $G$ and let $n$ be the number of leaves of $G$. Note that $n = \mu(G)$ so $n \geq 2r - 1$.  Let $\A \subseteq I^{(r)}(G)$ be any pairwise intersecting family. Define a partition $\B = \{A \in \A: c \not\in A\}$ and $\C = \{A \in \A: c \in A\}$. 

Notice that each vertex in each member of $\B$ is either a leaf or the neighbour of a leaf.  For $B \in \B$, let $M_B$ be the set of $r$ leaves that each either belongs to $B$ or is adjacent to a vertex in $B$.  We say that $M_B$ \emph{represents} $B$.   Let $\M = \{M_B : B \in \B\}$.  Note that each member of $\M$ might represent many different members of $\B$.  In fact, consider $M \in \M$.  It can represent any independent set that, for each leaf $\ell \in M$, contains either $\ell$ or its unique neighbour.  There are $2^r$ such sets but they can be partitioned into complementary pairs so, as $\B$ is pairwise intersecting, the number $s_M$ of members of $\B$ that $M$ represents is at most $2^{r-1}$.  We also note that $\M$ is pairwise intersecting (since $\B$ is pairwise intersecting).  

We have that 
\begin{equation} \label{ez}
|\B| = \sum_{M \in \M}{s_M} \leq \binom{n-1}{r-1}2^{r-1},
\end{equation}
where the inequality follows from Theorem~\ref{thm:meyer}. 

For $B \in \B$, let $N_B$ be the set of $n-r$ leaves that neither belong to $B$ nor are adjacent to a vertex in $B$.  Notice that $M_B$ and $N_B$ partition the set of leaves. Let $\N = \{N_B : B \in \B\}$.  For any pair $B_1, B_2 \in \B$, we know that $M_{B_1}$ and $M_{B_2}$ intersect, so $|M_{B_1} \cup M_{B_2}| \leq 2r -1$.  The leaves not in this union are members of both $N_{B_1}$ and $N_{B_2}$
 and there are at least $n-(2r-1) \geq 0$ of them. 
 Thus we can apply Lemma~\ref{shadow} to~$\N$ with $a=n-r$, $b=n-(2r-1)$ to obtain
\begin{equation}\label{e2}
|\N| \leq |\partial_{r-1} \N|. 
\end{equation}
Notice that, by definition, $\partial_{r-1} \N$ is a collection of sets of $r-1$ leaves each of which is, for some $B \in \B$, a subset of $N_B$ and is therefore disjoint to $M_B$ and so certainly does not intersect $B$.

Let us try to bound the size of $\C$.  Each $C \in \C$ contains a distinct set of~$r-1$ leaves.  We know this set must intersect every member of $\B$ so it cannot be a member $\partial_{r-1} \N$.  Thus we find
\begin{equation}\label{ec1}
|\C| \leq \binom{n}{r-1} - |\partial_{r-1} {\N}|. 
\end{equation}
We apply (\ref{e2}) to (\ref{ec1}) and note that $|\N| = |\M|$ to obtain
\begin{equation}\label{ea}
|\C| \leq \binom{n}{r-1} - |\M|. 
\end{equation}

Since $s_M \leq 2^{r-1}$ for each $M \in \M$, equality holds in (\ref{ez}) only if $|\M| \geq \binom{n-1}{r-1}$. Thus combining (\ref{ez}) and (\ref{ea}):
\begin{eqnarray}\label{eq:final}
|\A| & =& |\B| + |\C| \nonumber  \\ & \leq &  \sum_{M \in \M}{s_M} + \binom{n}{r-1} - |\M| \nonumber\\ 
& \leq & \binom{n-1}{r-1} 2^{r-1} + \binom{n}{r-1} - \binom{n-1}{r-1} \nonumber \\ &
= & \binom{n-1}{r-1} 2^{r-1} + \binom{n-1}{r-2}.
\end{eqnarray}

This proves that $G$ is $r$-EKR by Lemma~\ref{starlm}. We now show that $G$ is strictly $r$-EKR. If $r=n$ then $r=1$ so the result trivially holds. Suppose $r < n$. Then, by Theorem~\ref{thm:meyer}, equality holds in (\ref{ez}) and therefore also in (\ref{eq:final}) only if $\B$ is an $r$-star centred at a leaf $x$ or a neighbour $y$ of a leaf. It follows easily that $\C = \emptyset$ if $\A=\mathcal{I}_y^{(r)}(G)$; thus $\A = \mathcal{I}_x^{(r)}(G)$ as desired. \end{proof}
 
 \begin{remark*}
We demonstrate that if $G$ is a depth-two claw with $n$ leaves, then $G$ is not $n$-EKR by describing a pairwise intersecting family that is larger than the largest $n$-star. Let $c$ be the root of $G$ and let $G' = G - c$, a graph containing~$n$ copies of $K_2$ each of which contains one leaf of $G$.  Clearly $G'$ contains $2^n$ independent sets of size $n$ which can be partitioned into complementary pairs.  Let $\B$ be a family of $2^{n-1}$ independent sets of size $n$ formed by considering each complementary pair and choosing either the one that contains the greater number of leaves of $G$, or, if they each contain half the leaves, choosing one arbitrarily.  Notice that $\B$ is pairwise intersecting but is not a star. Let $\C = \{C \in \I^{(n)}(G): c \in C\}$. Clearly, $|\C| = \binom{n}{n-1}=n$ and for each pair $C \in \C$, $B \in \B$, we have that $C \cap B \not= \emptyset$.  Thus if $\A=\B \cup \C$, then $\A$ is pairwise intersecting, maximal and $|\A| = |\B| + |\C| = 2^{n-1} + n$. By Lemma~\ref{mainstar},  $\A$ has one more element than the largest $n$-star in $G$.
\end{remark*}

The above remark together with Theorem~\ref{mainthm} motivates the following conjecture.

\begin{conjecture}
Let $n$ and $r$ be positive integers, $n > r$, and let $G$ be a depth-two claw with $n$ leaves. Then $G$ is $r$-EKR.
\end{conjecture}

\section{Elongated Claws with Short Limbs}\label{auxsec} 

In this section we will prove Theorem~\ref{starthm}. We require some terminology and lemmas. For a vertex $v$ of a graph $G$,  let $G - v$ denote the graph obtained by deleting $v$ and incident edges from $G$, and let $G \del v$ be the graph obtained from $G$ by deleting the vertex $v$ and all its neighbours and their incident edges.

The following lemma has essentially the same proof as Lemma 2.5 in~\cite{kamat}, but we include a proof for completeness. 

\begin{lemma}\label{vlem}
Let $r$ be a positive integer, and let $G$ be a graph.  Let $v$ be a vertex of $G$ and let $u$ be a vertex of $G \del v$. Then 
\[
|\mathcal{I}^{(r)}_u(G)| =  |\mathcal{I}^{(r)}_u(G -v)| +  |\mathcal{I}^{(r-1)}_u(G \del v)|.
\]
\end{lemma}

\begin{proof}
Define a partition of $\mathcal{I}^{(r)}_u(G)$: $\B = \{A \in \mc{I}^r_u(G): v\notin A\}$ and $\C = \{A \in \mc{I}^{(r)}_u(G): v\in A\}$. Observe that $|\B| =  |\mathcal{I}^{(r)}_u(G - v)|$ and $|\C| = |\mathcal{I}^{(r-1)}_u(G \del v)|$. This implies the lemma. 
\end{proof}

\begin{lemma}\label{lstar}
Let $r$ be a positive integer and let $G$ be an elongated claw with a short limb with root $c$. If $x$ is a leaf of $G$ adjacent to $c$, then $x$ is the centre of a largest $r$-star of $G$.
\end{lemma}

\begin{proof}
Let $v$ be a vertex in $G$ that is not a leaf adjacent to $c$.  We must show that $\I^{(r)}_v(G)$ is no larger than $\I^{(r)}_x(G)$.  If $v=c$ this is immediate since $\{A \setminus \{c\} \cup \{x\} : A \in \I^{(r)}_c(G)  \}$ has the same cardinality as $\I^{(r)}_c(G)$ and is a subset of $\I^{(r)}_x(G)$.

If $v \neq c$, let  $L$ be the limb of $G$ that contains $v$.  To prove the lemma, we find an injection $f$ from $\I^{(r)}_v(G)$ to $\I^{(r)}_x(G)$.  Let $A \in \I^{(r)}_v(G)$. We distinguish a number of cases. 
\begin{enumerate}
\item If $x \in A$, then $f(A) = A$. 
\item If $x \not\in A$ and $c \not\in A$, then $f(A) = A\backslash\{v\} \cup \{x\}$.
\item If $x \not\in A$ and $c \in A$, let $X =\{v=x_1, \dots, x_m\}$ be the set of vertices from~$v$ towards the neighbour $x_m$ of $c$ in $L$. Let $Y = A \cap X = \{x_{i_1}, \dots, x_{i_j}\}$ for some $m > j \geq 1$. Let $Z = \{x_{i_1+1}, \dots, x_{i_j+1}\}$ and observe that $|Y| = |Z|$. Then $f(A) = (A \cup Z \cup \{x\}) \backslash (Y \cup \{c\})$. 
\end{enumerate}

\noindent It can be verified that $f$ is injective as required. 
\end{proof}

We now prove Theorem~\ref{starthm} using an approach based on that of the proof of~\cite[Theorem 1.22]{kamat}. 

\begin{proof}[Proof of Theorem~\ref{starthm}]
Let $c$ be the root of $G$. Let $\A \subseteq \indr$ be any pairwise intersecting family.  We must show that $\A$ is no larger than the largest $r$-star. We use induction on $r$. If $r=1$ the result is true so suppose that $r \geq 2$ and that the result is true for smaller values of~$r$. 

We now use induction on the number of vertices in $G$.  The base case is that $G$ contains only the root and $n$ leaves; that is, $G=K_{1,n}$ and so the result follows from Theorem~\ref{starekr}.  So suppose that the number of vertices is at least~$n+2$ and that the result is true for graphs with fewer vertices.  

Let $x$ be a leaf adjacent to $c$. Let $v$ be a leaf that is not adjacent to $c$.  Let~$w$ be the unique neighbour of $v$ and let $z$ denote the other neighbour of~$w$. 

Define  $f:\A \to \indr$ such that for each $A \in \A$
\[
f(A)=\left\{\begin{array}{ll}A \backslash \{v\} \cup \{w\}, & v \in A,z \not\in A, A \backslash \{v\} \cup \{w\} \not\in \A\\ A,& \mbox{otherwise.}\end{array}\right.
\]

Define the families:
\begin{gather*}
\A' = \{f(A): A \in \A\}, \\
\B  = \{A \in \A': v \not\in A \},\\ 
\C = \{A \backslash \{v\} : v \in A, A \in \A'\}.
\end{gather*}
Notice that
\begin{equation}\label{eqmain}
|\A| = |\A'| = |\B| + |\C|. 
\end{equation}

\begin{claim}\label{claim1}
Each of $\B$ and $\C$ is pairwise intersecting.
\end{claim}

\begin{proof}
By the definition of $f$, we can partition $\B$ into $\B_1 = \{B \in \B : B \in \A\}$ and $\B_2 = \{B \in \B : B \setminus \{w\} \cup \{v\} \in \A \}$.   Then $\B_1$ is pairwise intersecting (since $\A$ is intersecting) and $\B_2$ is pairwise intersecting as every member contains $w$.  Next consider $B_1 \in \B_1$ and $B_2 \in \B_2$.   As $B_1$ and $B_2 \setminus \{w\} \cup \{v\}$ are both in $\A$ they intersect and this intersection does not contain $v$ (since it is not in $B_1$) so is a superset of $B_1 \cap B_2$.  So $\B$ is intersecting. 

By definition, if $C \in \C$, then $C \cup \{v\}$ is in $\A'$ and, by the definition of $f$, also in $\A$.  Using the definition of $f$ again, we must have that either $z$ is in~$C$, or $C  \cup \{w\}$ is in $\A$.  Let $C_1$ and $C_2$ be two members of $\C$.  Then either they both contain $z$ or if one of them, say $C_1$, does not, then $C_1 \cup \{w\}$ is in~$\A$.  As $C_2 \cup \{v\}$ is also in $\A$ and $\A$ is intersecting, we have that $C_1 \cup \{w\}$ and $C_2 \cup \{v\}$ must intersect.  By the independence of the two sets, this intersection contains neither $v$ nor $w$ and so $C_1$ and $C_2$ must intersect. The claim is proved.
\end{proof}

Note that $G - v$ is an elongated claw with a short limb, fewer vertices than~$G$ and with $n$ leaves.  We also note that each member of $\B$ contains $r$ vertices of $G-v$ and, by Claim~\ref{claim1}, $\B$ is pairwise intersecting.  By the induction hypothesis, $G-v$ is $r$-EKR and so the largest intersecting families are $r$-stars, and, by Lemma~\ref{lstar}, $\mathcal{I}^{(r)}_x(G -v)$ is a largest $r$-star of $G-v$.  Hence
\begin{equation}\label{eqb}
|\B| \leq |\mathcal{I}^{(r)}_x(G -v)|.
\end{equation}

Note that $G \del v$ is an elongated claw with a short limb, fewer vertices than~$G$ and with either $n$ or $n-1$ leaves. We also note that each member of~$\C$ contains $r-1$ vertices of $G \del v$ and, by Claim~\ref{claim1}, $\C$ is pairwise intersecting.  By the induction hypothesis, $G \del v$ is $(r-1)$-EKR and so the largest intersecting families are $(r-1)$-stars, and, by Lemma~\ref{lstar}, $\mathcal{I}^{(r-1)}_x(G \del v)$ is a largest $(r-1)$-star of $G \del v$.  Hence
\begin{equation}\label{eqc}
|\C| \leq |\mathcal{I}^{(r-1)}_x(G \del v)|.
\end{equation}

Combining (\ref{eqmain}), (\ref{eqb}) and (\ref{eqc}) and applying Lemma~\ref{vlem}:
\begin{eqnarray*}
|\A|& =& |\B|+|\C| \\  &\leq& 
|\mathcal{I}^{(r)}_x(G - v)| + |\mathcal{I}^{(r-1)}_x(G \del v)| \\ &=& |\mathcal{I}^{(r)}_x(G)|
\end{eqnarray*} 
and the theorem is proved. 
\end{proof}

\section{Centres of Largest $r$-stars in Trees} 

In this section, we prove Theorem~\ref{centrethm}.  We shall do this by defining another family of trees.    Let $n$, $k$ and $a$ be positive integers.  A $(k,a)$-claw is an elongated claw with $k$ limbs each of length $a$.  The tree $T^{n,k,a}$ contains, as induced subgraphs, $n$ disjoint $(k,a)$-claws, and one further vertex, the \emph{root} of $T^{n,k,a}$, that is joined by an edge to the root of each $(k,a)$-claw.  Figure~\ref{superclawfig} shows  $T^{5,2,3}$ as an example.  We note that Baber~\cite{baber} and Borg~\cite{borg3} showed that Conjecture~\ref{c2} is false by considering $T^{2,k,2}$.

\tikzstyle{vertex}=[circle,draw=black, minimum size=6pt, inner sep=0pt]
\tikzstyle{edge} =[draw,thick,-,black]

\begin{figure}[h]
\begin{center}

\begin{tikzpicture}[scale=0.5]

\node[vertex, fill=black]  (root) at (8,3) {};

\begin{scope}     
    \foreach \pos/ \name in {(0,0)/first, (-1,-2)/a, (-1,-4)/b, (-1,-6)/c, (1,-2)/d, (1,-4)/e, (1,-6)/f}
        \node[vertex, fill=black]  (\name) at \pos {};
        
\path[edge] (0,0) --  (-1,-2);
\path[edge] (0,0) --  (1,-2);
 \path[edge] (-1,-4) --  (-1,-2);
 \path[edge] (-1,-4) --  (-1,-6);
 \path[edge] (1,-2) --  (1,-4);
 \path[edge] (1,-2) --  (1,-6);

 \end{scope}

\begin{scope}[xshift=4cm] 
    \foreach \pos/ \name in {(0,0)/second, (-1,-2)/a, (-1,-4)/b, (-1,-6)/c, (1,-2)/d, (1,-4)/e, (1,-6)/f}
        \node[vertex, fill=black]  (\name) at \pos {};
        
\path[edge] (0,0) --  (-1,-2);
\path[edge] (0,0) --  (1,-2);
 \path[edge] (-1,-4) --  (-1,-2);
 \path[edge] (-1,-4) --  (-1,-6);
 \path[edge] (1,-2) --  (1,-4);
 \path[edge] (1,-2) --  (1,-6);

\end{scope}

\begin{scope}[xshift=8cm] 
    \foreach \pos/ \name in {(0,0)/third, (-1,-2)/a, (-1,-4)/b, (-1,-6)/c, (1,-2)/d, (1,-4)/e, (1,-6)/f}
        \node[vertex, fill=black]  (\name) at \pos {};
        
\path[edge] (0,0) --  (-1,-2);
\path[edge] (0,0) --  (1,-2);
 \path[edge] (-1,-4) --  (-1,-2);
 \path[edge] (-1,-4) --  (-1,-6);
 \path[edge] (1,-2) --  (1,-4);
 \path[edge] (1,-2) --  (1,-6);
        \end{scope}

\begin{scope}[xshift=12cm] 
    \foreach \pos/ \name in {(0,0)/fourth, (-1,-2)/a, (-1,-4)/b, (-1,-6)/c, (1,-2)/d, (1,-4)/e, (1,-6)/f}
        \node[vertex, fill=black]  (\name) at \pos {};
        
\path[edge] (0,0) --  (-1,-2);
\path[edge] (0,0) --  (1,-2);
 \path[edge] (-1,-4) --  (-1,-2);
 \path[edge] (-1,-4) --  (-1,-6);
 \path[edge] (1,-2) --  (1,-4);
 \path[edge] (1,-2) --  (1,-6);
        \end{scope}

\begin{scope}[xshift=16cm] 
    \foreach \pos/ \name in {(0,0)/fifth, (-1,-2)/a, (-1,-4)/b, (-1,-6)/c, (1,-2)/d, (1,-4)/e, (1,-6)/f}
        \node[vertex, fill=black]  (\name) at \pos {};
        
\path[edge] (0,0) --  (-1,-2);
\path[edge] (0,0) --  (1,-2);
 \path[edge] (-1,-4) --  (-1,-2);
 \path[edge] (-1,-4) --  (-1,-6);
 \path[edge] (1,-2) --  (1,-4);
 \path[edge] (1,-2) --  (1,-6);
        \end{scope}

       \path[edge] (0,0) --  (root);
       \path[edge] (4,0) --  (root);
       \path[edge] (8,0) --  (root);
       \path[edge] (12,0) --  (root);
       \path[edge] (16,0) --  (root);

\end{tikzpicture}
\end{center}

\caption{$T^{5,2,3}$}  \label{superclawfig}

\end{figure}

The key to proving Theorem~\ref{centrethm} is to show that, for certain values, the largest $r$-star of $T^{n,k,a}$ is centred at its root.  Let $\mathcal{I}_{\mbox{\scriptsize root}} (T^{n,k,a})$ be the family of independent sets of $T^{n,k,a}$ that contain its root, and let $\mathcal{I}_{\mbox{\scriptsize leaf}} (T^{n,k,a})$ be the number of independent sets of $T^{n,k,a}$ that contain a particular leaf (note that, by symmetry, the size of this family does not depend on which leaf we choose).   Notice that in these definitions, we are considering independent sets of all possible sizes.  In Lemma~\ref{lm:ce}, we will think of $|\mathcal{I}_{\mbox{\scriptsize root}} (T^{n,k,a})|$ and $|\mathcal{I}_{\mbox{\scriptsize leaf}} (T^{n,k,a})|$ as sequences indexed by $k$ with fixed $n$ and $a$.

First we need some further definitions and lemmas.  Let $a$ be a nonnegative integer.  Let $P_a$ denote the path on $a$ vertices.  Let $F(a)$ denote the number of independent sets in $P_a$  (notice that the empty set is an independent set of any graph).  We state without proof two simple observations.

\begin{lemma} \label{fa}
$F(0)=1$, $F(1)=2$ and, for $a \geq 2$, $F(a) = F(a-1) + F(a-2)$.  Moreover, for $a \geq 3$, each vertex of degree 1 in $P_a$ belongs to more independent sets that each vertex of degree 2. 
\end{lemma}

\noindent
We notice that $(F_a)$ is, of course, the Fibonacci sequence (without the initial term).

We now prove a simple result about $(k,a)$-claws that we will use later.

\begin{lemma} \label{kaclaws}
Let $n$, $k$ and $a$ be positive integers, let $b$ be a nonnegative integer and let $G$ be a graph that contains $n-1$ disjoint $(k,a)$-claws.  The number of independent sets of $G$ that each contain the roots of at least $b$ of the $(k,a)$-claws is
\[
\sum_{i=b}^{n-1} \binom{n-1}{i}F(a-1)^{ik}F(a)^{(n-1-i)k}.
\]
\end{lemma}
\begin{proof}
Note that a $(k,a)$-claw with its root removed is $k$ disjoint copies of~$P_a$ and so the claw contains $F(a)^k$ independent sets that do not include the root.  Similarly it contains $F(a-1)^k$ independent sets that do include the root (in this case one considers the graph obtained when the root and its neighbours are removed).  This kind of argument recurs many times in this section; we use it first to complete the proof of the lemma.

Each summand is the number of independent sets that contain \emph{exactly} $i$ of the roots:  the three factors count the number of ways of choosing~$i$ claws (whose roots will be in the independent set), the number of independent sets in those chosen claws (given that their roots are included) and the number of independent sets in the unchosen claws (given that their roots are not included).  Then the sum is over the possible values of $i$. 
\end{proof}

One more definition.  For a vertex $v$ in a graph, we denote by $N(v)$ the set of vertices that are adjacent to $v$.  

\begin{lemma}\label{lm:ce}
Let $a$ and $n$ be positive integers, $n \geq 2$, $a \geq 2$.  Then
\begin{equation*}
\frac{|\mathcal{I}_{\mbox{\scriptsize root}} (T^{n,k,a})|}{|\mathcal{I}_{\mbox{\scriptsize leaf}} (T^{n,k,a})|} \to \frac{F(a-1) + F(a-2)}{2F(a-2)}    \textrm{ as } k \to \infty.
\end{equation*} 
\end{lemma}

\begin{proof}
Let $x$ be the root of $T^{n,k,a}$ and let $y$ be one of its leaves.  
We note that as $k$ is not fixed, we are concerned with finding properties not of a specific graph, but of the family of graphs $T^{n,k,a}$ for fixed $n$ and $a$.  So we might have written $x_k$ and $y_k$ to indicate that when we choose a particular vertex, we must first fix which graph in the family we are looking at.  For simplicity, we avoid this explicit notation throughout.

Some more notation to improve readability: let $I(x)=|\mathcal{I}_{\mbox{\scriptsize root}} (T^{n,k,a})|$ and $I(y)=|\mathcal{I}_{\mbox{\scriptsize leaf}} (T^{n,k,a})|$, and let $I(x,y)$ be the number of independent sets that contain both $x$ and $y$.

 We can say immediately that 
\begin{equation}\label{ixy}
 I(x,y) = F(a-2)F(a)^{nk-1}
 \end{equation}
as we just need to count the independent sets in the graph obtained from $T^{n,k,a}$ when $x$ and $y$ and their neighbours are removed and this graph 
contains $nk-1$ copies of $P_a$ and one copy of $P_{a-2}$.  Let $I'(x)$ and $I'(y)$ be the number of independent sets that contain $x$ but not $y$, and $y$ but not $x$, respectively; that is $I(x) = I'(x) + I(x,y)$ and $I(y) = I'(y) + I(x,y)$.
Let $T_x = T^{n,k,a} \setminus (N(x) \cup \{x, y\})$ and $T_y = T^{n,k,a} \setminus (N(y) \cup \{x, y\})$. So $I'(x)$ is the number of independent sets in $T_x$ and $I'(y)$ is the number of independent sets in $T_y$.  As
 $T_x$ consists of $nk - 1$ disjoint copies of $P_a$ and one copy of $P_{a-1}$, we have 
 \begin{equation}\label{ix}
 I'(x) =  F(a-1) F(a)^{nk-1}. 
 \end{equation}
 
Evaluating $I'(y)$ will require a little more work. Notice that $T_y$ contains $n-1$ disjoint $(k, a)$-claws and one elongated claw $C$ that has $k-1$ limbs of length $a$ and one limb of length $a-2$.  Let $R$ be the set of roots of the $(k,a)$-claws and let $c$ denote the root of $C$.   We define a partition of the independent sets of $T_y$:
 \begin{itemize}
\item $S_1$ is the family of independent sets that do not contain any member of $R$ nor $c$.
\item $S_2$ is the family of independent sets that contain $c$.
\item $S_3$ is the family of independent sets that do not contain $c$ but do intersect $R$.
  \end{itemize}
So $S_1$ contains independent sets of $T_y \setminus (R \cup \{c\})$, a graph that consists of $nk-1$ disjoint copies of $P_a$ and one copy of $P_{a-2}$.  
Thus we have, using also~(\ref{ixy}),
  \begin{equation}\label{s1}
 |S_1| = F(a-2)F(a)^{nk-1} = I(x,y).
 \end{equation}
We will need the following observation:
\begin{equation} \label{ixs1}
\frac{|S_1|}{I'(x)} = \frac{F(a-2)}{F(a-1)}.
\end{equation}

Next we note that $S_2$ contains independent sets in $T_y$ that contain $c$ so to find its size we count the number of independent sets in $T_y \setminus (N(c) \cup  \{c\})$, a graph that contains $k-1$ copies of $P_{a-1}$, $n-1$ disjoint $(k,a)$-claws 
and one copy of $P_{a-3}$ (if $a \geq 3$) or one copy of $P_{a-2}$ (if $a=2$).  Thus, noting that $F(a-3) \leq F(a-2)$, we have
   \[|S_2| \leq F(a-2)F(a-1)^{k-1} \sum _{i=0}^{n-1} \binom{n-1} {i} F(a-1)^{ik}F(a)^{(n-1-i)k}\]   
where the sum is the number of independent sets in $n-1$ disjoint $(k,a)$-claws (by Lemma~\ref{kaclaws} with $b=0$).
Noting as before that $F(a-2)<F(a-1)$ and that, for all $i$, ${\binom{n-1}{i} \leq \binom{n-1} {\lfloor (n-1)/2 \rfloor}}$, we obtain
\begin{eqnarray*}
|S_2| &\leq & F(a-1)^{k} \sum _{i=0}^{n-1} \binom{n-1} {\lfloor (n-1)/2 \rfloor} F(a-1)^{ik}F(a)^{(n-1-i)k} \\
& = & F(a-1)^{k} F(a)^{(n-1)k}    \sum _{i=0}^{n-1} \binom{n-1} {\lfloor (n-1)/2 \rfloor} \left( \frac{F(a-1)}{F(a)}   \right)^{\!\!\!ik} \\
& \leq & F(a-1)^{k} F(a)^{(n-1)k}  \sum _{i=0}^{n-1} \binom{n-1} {\lfloor (n-1)/2 \rfloor}.  
\end{eqnarray*}
So we can write
\begin{equation} \label{s2}
|S_2| \leq c_2 F(a-1)^{k} F(a)^{(n-1)k}   
\end{equation}
where $c_2$ is a constant that does not depend on $k$.

Let us note now that, using (\ref{ix}) and (\ref{s2}), we have 
\begin{equation*} 
\frac{|S_2|}{I'(x)} \leq c_2 \frac{F(a)}{F(a-1)}  \left( \frac{F(a-1)}{F(a)}   \right)^{\!\!\!k}.
\end{equation*}
Hence, since for $a \geq 2$, $F(a-1) < F(a)$,
\begin{equation} \label{ixs2}
\frac{|S_2|}{I'(x)} \to 0 \textrm{ as } k \to \infty.
\end{equation}
And from (\ref{s1}) and (\ref{s2}), we have 
\begin{equation*} 
\frac{|S_2|}{|S_1|} \leq c_2 \frac{F(a)}{F(a-2)}  \left( \frac{F(a-1)}{F(a)}   \right)^{\!\!\!k}.
\end{equation*}
Hence
\begin{equation} \label{s21}
\frac{|S_2|}{|S_1|} \to 0 \textrm{ as } k \to \infty.
\end{equation}

As $S_3$ contains independent sets in $T_y$ that do not contain $c$ but intersect~$R$, we must count the number of independent sets in $T_y \setminus \{ c \}$ --- a graph that contains one copy of $P_{a-2}$, $k-1$ copies of $P_{a}$ and $n-1$ disjoint $(k,a)$-claws --- that contain the root of at least one of the $(k,a)$-claws. 
Thus we have
\[|S_3| =F(a-2)F(a)^{k-1} \sum _{i=1}^{n-1} \binom{n-1} {i} F(a-1)^{ik}F(a)^{(n-1-i)k}\]   
where the sum is the number of independent sets in $n-1$ disjoint $(k,a)$-claws that include at least one of the roots (by Lemma~\ref{kaclaws} with $b=1$).  Reasoning as before, we find
\begin{eqnarray*}
|S_3|& \leq  & F(a)^{k} \sum _{i=1}^{n-1} \binom{n-1}{\lfloor (n-1)/2 \rfloor} F(a-1)^{ik}F(a)^{(n-1-i)k} \\
& = &  F(a)^{nk} \sum _{i=1}^{n-1} \binom{n-1}{\lfloor (n-1)/2 \rfloor} \left(\frac{F(a-1)}{F(a)}\right)^{\!\!\!ik} \\
& \leq &  F(a)^{nk} \sum _{i=1}^{n-1} \binom{n-1}{\lfloor (n-1)/2 \rfloor} \left(\frac{F(a-1)}{F(a)}\right)^{\!\!\!k}.
\end{eqnarray*}
Thus we obtain
\begin{equation} \label{s3}
|S_3| \leq c_3 F(a-1)^{k} F(a)^{(n-1)k}   
\end{equation}
where $c_3$ is a constant that does not depend on $k$.  Comparing (\ref{s2}) and (\ref{s3}), we see that the same arguments used to obtain (\ref{ixs2}) and (\ref{s21}) give us
\begin{eqnarray}
\frac{|S_3|}{I'(x)} & \to 0 & \textrm{ as } k \to \infty, \label{ixs3} \\
\frac{|S_3|}{|S_1|} &\to 0& \textrm{ as } k \to \infty. \label{s31}
\end{eqnarray}
We combine (\ref{ixs1}), (\ref{ixs2}) and (\ref{ixs3}) to find
\begin{equation*}
\frac{I'(y)+I(x,y)}{I'(x)} = \frac{2|S_1|}{I'(x)} + \frac{|S_2|}{I'(x)} +\frac{|S_3|}{I'(x)} \to \frac{2F(a-2)}{F(a-1)} \textrm{ as } k \to \infty. \label{iys1}
\end{equation*}
And from (\ref{s21}) and (\ref{s31}), we have
\begin{equation*}
\frac{I'(y)+I(x,y)}{I(x,y)} = \frac{2|S_1|}{|S_1|} + \frac{|S_2|}{|S_1|} +\frac{|S_3|}{|S_1|} \to 2 \textrm{ as } k \to \infty. \label{iys1}
\end{equation*}

\noindent Using these last two observations, we can complete the proof:
\begin{eqnarray*}
\frac{I(x)}{I(y)} & = & \frac{I'(x)}{I'(y)+I(x,y)} + \frac{I(x,y)}{I'(y)+I(x,y)}  \\
& \to & \frac{F(a-1)}{2F(a-2)} + \frac{1}{2} = \frac{F(a-1)+F(a-2)}{2F(a-2)} \textrm{ as } k \to \infty. \label{iys1}
\end{eqnarray*}
\end{proof}  

\begin{proof}[Proof of Theorem~\ref{centrethm}]
In the family of trees $T^{n,k,d-1}$, the root vertex has degree $n$ and is at distance $d$ from every leaf.  By Lemma~\ref{lm:ce}, for sufficiently large $k$, 
\[
|\mathcal{I}_{\mbox{\scriptsize root}} (T^{n,k,d-1})| >
|\mathcal{I}_{\mbox{\scriptsize leaf}} (T^{n,k,d-1})|.
\]
As $\mathcal{I}_{\mbox{\scriptsize root}} (T^{n,k,d-1})$ and $\mathcal{I}_{\mbox{\scriptsize leaf}} (T^{n,k,d-1})$ are each the (disjoint) union, over all positive integers $r$, of $r$-stars centred at, respectively, the root and the leaf, there must be some $r$ for which the $r$-star centred at the root is strictly larger than that centred at the leaf.

The theorem will follow if we can show that for any positive integer $r$, for any tree $T^{n,k,d-1}$, and for any vertex $w$ that is neither the root nor a leaf, the $r$-star centred at $w$ is no larger than a $r$-star centred at a leaf.

Let $x$ be the root of $T^{n,k,d-1}$.  Let $C$ be the component of $T^{n,k,d-1} \setminus \{x\}$ that contains $w$ and let $D$ be the union of the other components.  Noting that $C$ is an elongated claw, let $y$ be the leaf of the limb that contains $w$ (or any limb if $w$ is the root of $C$). Let $R(w)$ and $R(y)$ be the number of independents sets of $T^{n,k,d-1}$ of size $r$ that include $x$ and contain, respectively, $w$ and $y$.  Similarly let $S(w)$ and $S(y)$ be the number of independent sets that contain, respectively, $w$ and $y$, but that do not include $x$.  For $v \in \{w,y\}$, we can write 
\[
S(v) = \sum_{i=0}^r \mathcal{I}^{(i)}_v(C) \times \mbox{number of independent sets of size $r-i$ in $D$}. 
\]
By Lemma~\ref{starlm}, $\mathcal{I}^{(i)}_y(C) \geq \mathcal{I}^{(i)}_w(C)$ for all $i$, and, as the second term in the product does not depend on $v$, we have that $S(y) \geq S(w)$.  

Now we consider independent sets of size $r$ that do contain $x$.  These can be bijectively matched with independent sets of size $r-1$ in $T^{n,k,d-1} \setminus (N(x) \cup \{x\})$; this graph contains $nk$ copies of $P_{d-1}$.  If $w$ is the root of $C$ it is not in this graph, and in this case $R(w)=0$ and we are done.  In all other cases, $w$ and $y$ belong to the same copy of $P_{d-1}$ which we denote $P$.  The union of the other paths, we denote $Q$.
For $v \in \{w,y\}$, 
\[
R(v) = \sum_{i=0}^r \mathcal{I}^{(i)}_v(P) \times \mbox{number of independent sets of size $r-i$ in $Q$}. 
\]
As $y$ is a vertex of degree 1 in $P$, by Lemma~\ref{fa}, $v=y$ maximises $\mathcal{I}^{(i)}_v(P)$.  Again, the second term does not depend on $v$ so we have that $R(y) \geq R(w)$ and the proof is complete.
\end{proof}

\subsection{Further Counterexamples}

Let us finally remark that one can define a much broader class of trees with the property that the largest $r$-stars are not centred at leaves (which therefore provides further counterexamples to Conjecture~\ref{c2}) by, for example, taking copies of $T^{n,k,a}$ and adding an additional root vertex joined to the root of each $T^{n,k,a}$ --- and this process of duplicating and joining (via a new root) can be repeated ad infinitum.  Moreover, it does not, in fact, matter which trees are used to initialize this process: if the number of copies made is large enough a graph where the largest $r$-stars are not centred at leaves is obtained.  This does not ultimately add anything to the result stated in Theorem~\ref{centrethm} so we omit further details.

\bibliography{bibliography}{}
\bibliographystyle{abbrv}
\end{document}